\documentclass[11pt,letterpaper]{amsart}

\usepackage[T1]{fontenc}
\usepackage[utf8]{inputenc}
\usepackage{amsmath,amssymb,amsthm,mathtools}
\usepackage{microtype}
\usepackage{mathpazo}
\usepackage[margin=1.08in]{geometry}
\usepackage{enumitem,booktabs,array,longtable}
\usepackage{cite}
\usepackage{xcolor}
\usepackage[colorlinks=true,linkcolor=black,citecolor=black,urlcolor=black,
pdftitle={Binary dimension-free discretization and complexification on real cubes},
pdfsubject={Exact transfer between binary discretization and polynomial complexification},
pdfkeywords={dimension-free discretization, polynomial complexification,
Remez inequalities, Boolean cube, multiaffine polynomials, Chebyshev polynomials}
]{hyperref}

\allowdisplaybreaks
\numberwithin{equation}{section}
\setlist{itemsep=.25em,topsep=.45em}

\newtheorem{theorem}{Theorem}[section]
\newtheorem{maintheorem}{Theorem}

\newtheorem{lemma}[theorem]{Lemma}

\theoremstyle{remark}

\newcommand{\C}{\mathbb C}
\newcommand{\R}{\mathbb R}
\newcommand{\T}{\mathbb T}
\newcommand{\N}{\mathbb N}
\newcommand{\D}{\mathbb D}

\title[Binary discretization and cube complexification]
{Binary dimension-free discretization and polynomial complexification on real cubes}

\author[L.~M.~Casta\~no-Mar\'in]{Luis Miguel Casta\~no-Mar\'in}
\address{Departamento de Matem\'aticas\\
Universidad Nacional de Colombia\\
Bogot\'a, Colombia}
\email{lucastanom@unal.edu.co}

\author[D.~N\'u\~nez-Alarc\'on]{Daniel N\'u\~nez-Alarc\'on}
\address{Departamento de Matem\'aticas\\
Universidad Nacional de Colombia\\
Bogot\'a, Colombia}
\email{dnuneza@unal.edu.co}

\author[J.~Santos]{Joedson Santos}
\address{Departamento de Matem\'atica\\
Universidade Federal da Para\'iba\\
Centro de Ci\^encias Exatas e da Natureza, Jo\~ao Pessoa, PB, Brazil}
\email{joedson.santos@academico.ufpb.br}


\begin{document}
\begin{abstract}
Let $C(d,2)$ be the optimal dimension-free ratio between the polytorus
and Boolean-cube norms of complex multiaffine polynomials of degree at
most $d$. For every $d$, $C(d,2)$ equals the unrestricted
complexification constant of the real cube for complex polynomials of
degree at most $d$. Moreover,
$\lim_{d\to\infty}C(d,2)^{1/d}=1+\sqrt2$, and for every $d\ge3$,
$\beta(1-6/(5d))(1+\sqrt2)^d\le C(d,2)\le(1+\sqrt2)^d$,
where
$\beta=15\pi/[2(5\sqrt{26}+\log(5+\sqrt{26}))]
=0.8473222863\ldots$.
\end{abstract}

\maketitle

\section{Introduction and main results}
Becker, Klein, Slote, Volberg, and Zhang \cite{BKSVZ} studied dimension-free discretization of polynomial uniform norms by product sets whose size is controlled by the maximal individual degree rather than by the total degree.  Write $\D=\{z\in\C:|z|\le1\}$ and $\T=\{z\in\C:|z|=1\}$.
For a polynomial $P:\C^N\to\C$ and a nonempty compact set $E\subset\C^N$, put
$\|P\|_E=\max_{z\in E}|P(z)|$.
For analytic polynomials of total degree at most $d$ and individual degree at most $K-1$, they proved dimension-free estimates of the form
\begin{equation}\label{eq:general-discretization}
 \|P\|_{\D^N}\le C(X)^d\|P\|_{X^N},
\end{equation}
where $X\subset\D$ has $K$ points and the constant is independent of $N$.  They also showed that exponential dependence on $d$ is unavoidable for product sampling sets of this minimal cardinality.  For products of $K$th roots of unity, this work continues the dimension-free Remez inequality of Slote, Volberg, and Zhang \cite{SloteRemez}.
Related dimension-free Remez phenomena for analytic functions were developed by Nazarov, Sodin, and Volberg \cite{NSV2003}.

For the smallest nontrivial sampling size, $K=2$,
\[
 \Omega_2=\{-1,1\}.
\]
The relevant polynomials have individual degree at most one.  Iterating the one-variable maximum modulus principle \cite[Chap.~VII]{Conway} in the $N$ coordinates gives $\|P\|_{\D^N}=\|P\|_{\T^N}$ for every analytic polynomial $P:\C^N\to\C$.  For a multiindex $\alpha=(\alpha_1,\ldots,\alpha_N)$, put
$|\alpha|=\alpha_1+\cdots+\alpha_N$ and
$z^\alpha=z_1^{\alpha_1}\cdots z_N^{\alpha_N}$.  For integers $N\ge d\ge1$, let
\[
 \mathcal P_{N,d}^{\mathrm{mult}}
 :=\left\{P:\C^N\to\C:
 P(z)=\sum_{\substack{\alpha\in\{0,1\}^N\\ |\alpha|\le d}}
 c_\alpha z^\alpha,\ c_\alpha\in\C\right\}.
\]
For the binary sampling set, the corresponding optimal dimension-free constant is
\begin{equation}\label{eq:uniform-transfer-constant}
 C(d,2):=\sup_{N\ge d}\ \sup_{0\ne P\in\mathcal P_{N,d}^{\mathrm{mult}}}
 \frac{\|P\|_{\T^N}}{\|P\|_{\{-1,1\}^N}}.
\end{equation}
Restriction to $\{-1,1\}^N$ identifies these polynomials with Fourier--Walsh polynomials of degree at most $d$.  Adding unused variables shows that the restriction $N\ge d$ does not change the constant.

For a multiaffine polynomial, the modulus is convex in each real coordinate,
and therefore
\begin{equation}\label{eq:multiaffine-real-cube-norm}
\|P\|_{[-1,1]^N}=\|P\|_{\{-1,1\}^N}.
\end{equation}
Consequently, every quotient in \eqref{eq:uniform-transfer-constant}
has denominator $\|P\|_{[-1,1]^N}$.

An early sharp result in a related direction is due to Visser
\cite{Visser}. If
\[
P=P_0+P_1+\cdots+P_d
\]
is a polynomial with real coefficients on $[-1,1]^N$, where $P_j$ is
$j$-homogeneous, then Visser's theorem gives
\begin{equation}\label{eq:visser-top-homogeneous}
\|(P_d)_{\C}\|_{\T^N}\le 2^{d-1}\|P\|_{[-1,1]^N}.
\end{equation}
The constant $2^{d-1}$ is sharp, already in one variable by the Chebyshev
polynomial $T_d:\C\to\C$, normalized by
\[
T_d(\cos\theta)=\cos(d\theta)\qquad(\theta\in\R).
\]
In \eqref{eq:visser-top-homogeneous}, only the top homogeneous part of a
real polynomial is complexified. The constant $C(d,2)$ concerns the full
multiaffine polynomial and allows complex coefficients.

Dimension-free comparisons between real and complex sup norms of polynomials
in many variables were studied explicitly by Aron, Beauzamy, and Enflo
\cite{AronBeauzamyEnflo}. In a Banach-space setting, Lacruz \cite{Lacruz}
formulated the associated norming constants for arbitrary polynomials and
related their asymptotic behavior to extremal functions and capacities. The corresponding constants are defined without multiaffinity or
Boolean sampling. Polynomial complexification on real Banach spaces is treated,
for instance, in \cite{MunozSarantopoulosTonge,DineenBook} and in the recent
work \cite{Rodriguez2025}. In several complex variables, pluripotential estimates
of Siciak \cite{Siciak1981} and Klimek \cite{Klimek} control polynomial growth
away from the real cube. In particular, for every
polynomial $H\in\C[t_1,\ldots,t_s]$ of total degree at most $d$, the estimate
recorded by Defant, Masty{\l}o, and P\'erez
\cite[Lemma~1.3(2)]{DefantMastyloPerez} gives
\begin{equation}\label{eq:Klimek-cube-complexification}
\|H\|_{\T^s}\le(1+\sqrt2)^d\|H\|_{[-1,1]^s}.
\end{equation}

The exponential base $1+\sqrt2$ in \eqref{eq:Klimek-cube-complexification}
already occurs in one variable. The Bernstein--Walsh inequality for $[-1,1]$; see, for example,
\cite[Chap.~5]{Ransford}, together with the Chebyshev polynomials, yields
\[
\lim_{d\to\infty}
\left(
\sup_{\substack{0\ne h\in\C[t]\\ \deg h\le d}}
\frac{\|h\|_{\D}}{\|h\|_{[-1,1]}}
\right)^{1/d}
=1+\sqrt2.
\]
For $d\ge1$, set
\begin{equation}\label{eq:Gamma-definition}
\Gamma_d:=
\sup_{s\ge1}\
\sup_{\substack{0\ne H\in\C[t_1,\ldots,t_s]\\ \deg H\le d}}
\frac{\|H\|_{\T^s}}{\|H\|_{[-1,1]^s}}.
\end{equation}
\begin{maintheorem}\label{thm:complexification-identity}
For every positive integer $d$,
\begin{equation}\label{eq:intro-unrestricted-complexification}
C(d,2)=\Gamma_d.
\end{equation}
Equivalently,
\[
C(d,2)=
\sup_{s\ge1}\
\sup_{\substack{0\ne H\in\C[t_1,\ldots,t_s]\\ \deg H\le d}}
\frac{\|H\|_{\T^s}}{\|H\|_{[-1,1]^s}}.
\]
\end{maintheorem}

Combining Theorem~\ref{thm:complexification-identity} with
\eqref{eq:Klimek-cube-complexification} gives
\begin{equation}\label{eq:known-binary-remez}
C(d,2)\le(1+\sqrt2)^d.
\end{equation}
On the other hand, taking $s=1$ and $H=T_d$ in
Theorem~\ref{thm:complexification-identity} gives
\[
C(d,2)\ge |T_d(i)|
=\frac12\left|
\bigl(i(1+\sqrt2)\bigr)^d+
\bigl(i(1-\sqrt2)\bigr)^d
\right|.
\]
The one-variable Chebyshev test gives the normalized asymptotic lower
factor $1/2$. Theorem~\ref{thm:sharp-binary-cost} gives the stronger uniform
bound
\[
  C(d,2)\ge
  \beta\left(1-\frac{6}{5d}\right)(1+\sqrt2)^d,
  \qquad d\ge3,
\]
with \(\beta>1/2\).

Set
\begin{equation}\label{eq:asymptotic-beta}
\beta:=
\frac{15\pi}{2\bigl(5\sqrt{26}+\log(5+\sqrt{26})\bigr)}
\approx0.8473222863>\frac\pi4.
\end{equation}

\begin{maintheorem}\label{thm:sharp-binary-cost}
For every integer $d\ge3$,
\begin{equation}\label{eq:improved-exponential-sandwich}
\beta\left(1-\frac6{5d}\right)(1+\sqrt2)^d
\le C(d,2)\le(1+\sqrt2)^d.
\end{equation}
In particular,
\begin{equation}\label{eq:improved-normalized-limits}
\beta\le
\liminf_{d\to\infty}\frac{C(d,2)}{(1+\sqrt2)^d}
\le
\limsup_{d\to\infty}\frac{C(d,2)}{(1+\sqrt2)^d}
\le1,
\end{equation}
and
\begin{equation}\label{eq:sharp-exponential-rate}
\lim_{d\to\infty}C(d,2)^{1/d}=1+\sqrt2.
\end{equation}
\end{maintheorem}

The existence and value of
$\displaystyle\lim_{d\to\infty}C(d,2)/(1+\sqrt2)^d$ remain open.

\section{Auxiliary lemmas}
\label{sec:auxiliary-lemmas}

\subsection{Symmetric multiaffine lifts}
The Chebyshev polynomial of the first kind $T_d$ is normalized by \cite[Chap.~1]{MasonHandscomb}
\[
 T_d(\cos\theta)=\cos(d\theta)\qquad(\theta\in\R).
\]
For $n\in\mathbb N$ and $0\le k\le n$, define $U_{n,k}:\C^n\to\C$ by
\begin{equation}\label{eq:normalized-elementary-symmetric}
 U_{n,0}(z):=1,\qquad
 U_{n,k}(z):=\binom nk^{-1}
 \sum_{\substack{S\subseteq\{1,\ldots,n\}\\|S|=k}}\prod_{j\in S}z_j
 \quad(k\ge1).
\end{equation}
Define also $m_n:\{-1,1\}^n\to[-1,1]$ by
\begin{equation}\label{eq:Boolean-empirical-mean}
 m_n(x):=\frac1n\sum_{j=1}^n x_j.
\end{equation}

\begin{lemma}\label{lem:symmetric-lift-approximation}
For every $n\ge1$, every $0\le k\le n$, and every $x\in\{-1,1\}^n$,
\begin{equation}\label{eq:elementary-power-comparison}
 |U_{n,k}(x)-m_n(x)^k|\le\frac{k(k-1)}n.
\end{equation}
\end{lemma}

\begin{proof}
For $k=0$ and $k=1$, the two sides compared in
\eqref{eq:elementary-power-comparison} are equal. Fix $2\le k\le n$.
Let
\[
 \Omega_{n,k}:=\{1,\ldots,n\}^k
\]
with normalized counting measure
\[
 d\mu(i_1,\ldots,i_k):=\frac1{n^k},
\]
and define
\[
 Y(i_1,\ldots,i_k):=\prod_{r=1}^k x_{i_r},
 \qquad
 D_k:=\{(i_1,\ldots,i_k)\in\Omega_{n,k}: i_1,\ldots,i_k
 \text{ are pairwise distinct}\}.
\]
Writing the normalized sum as an integral,
\begin{align}
 \int_{\Omega_{n,k}}Y\,d\mu
 &=\frac1{n^k}\sum_{i_1=1}^n\cdots\sum_{i_k=1}^n
      x_{i_1}\cdots x_{i_k}\notag\\
 &=\prod_{r=1}^k\left(\frac1n\sum_{j=1}^n x_j\right)
 =m_n(x)^k.
 \label{eq:counting-integral-power}
\end{align}
On $D_k$, each $k$-element subset of $\{1,\ldots,n\}$ occurs in exactly
$k!$ ordered forms. Hence
\begin{equation}\label{eq:counting-integral-symmetric}
 \frac1{\mu(D_k)}\int_{D_k}Y\,d\mu=U_{n,k}(x).
\end{equation}
Write $\delta:=\mu(D_k)$. Since $|Y|=1$, the averages of $Y$ over $D_k$
and $D_k^c$ lie in $[-1,1]$. Splitting the integral in
\eqref{eq:counting-integral-power} over $D_k\cup D_k^c$ and using
\eqref{eq:counting-integral-symmetric}, one obtains
\begin{align}
 \left|U_{n,k}(x)-m_n(x)^k\right|
 &=(1-\delta)\left|
 \frac1{\mu(D_k)}\int_{D_k}Y\,d\mu
 -\frac1{\mu(D_k^c)}\int_{D_k^c}Y\,d\mu\right|\notag\\
 &\le 2\mu(D_k^c).
 \label{eq:conditioning-error}
\end{align}
If a point of $\Omega_{n,k}$ lies in $D_k^c$, then $i_r=i_s$ for at least
one pair $1\le r<s\le k$. For each fixed pair $(r,s)$, exactly $n^{k-1}$
of the $n^k$ points satisfy $i_r=i_s$. Therefore
\begin{equation}\label{eq:collision-union-bound}
 \mu(D_k^c)
 \le\sum_{1\le r<s\le k}\mu\{i_r=i_s\}
 =\binom{k}{2}\frac1n.
\end{equation}
Substituting \eqref{eq:collision-union-bound} into
\eqref{eq:conditioning-error} proves
\eqref{eq:elementary-power-comparison}.
\end{proof}

\subsection{Regular phase sums}
For integers $m\ge1$, set
\begin{equation}\label{eq:phase-midpoint-nodes}
 \theta_{m,j}:=-\frac\pi2+\left(j-\frac12\right)\frac\pi m,
 \qquad 1\le j\le m.
\end{equation}

\begin{lemma}\label{lem:regular-phase-sum}
For every $m\ge1$,
\begin{equation}\label{eq:regular-phase-max}
 \max_{u\in[-1,1]^m}
 \left|\sum_{j=1}^m e^{i\theta_{m,j}}u_j\right|
 =\frac1{\sin(\pi/(2m))}.
\end{equation}
Moreover,
\begin{equation}\label{eq:regular-polygon-cosine-sum}
 \sum_{j=1}^m e^{i\theta_{m,j}}
 =\sum_{j=1}^m\cos\theta_{m,j}
 =\frac1{\sin(\pi/(2m))}.
\end{equation}
\end{lemma}

\begin{proof}
The case $m=1$ is immediate. For $m\ge2$, maximizing a real projection leads to
\[
 \max_{u\in[-1,1]^m}\left|\sum_{j=1}^m e^{i\theta_{m,j}}u_j\right|
 =\max_{\phi\in\R}\sum_{j=1}^m|\cos(\theta_{m,j}-\phi)|.
\]
If none of these cosines vanishes, the maximizing signs select the $m$
vertices of the regular $2m$-gon
$\{\pm e^{i\theta_{m,j}}:1\le j\le m\}$ that lie in the corresponding open
half-plane. They are consecutive vertices, and their sum has modulus
\[
 \left|\sum_{k=0}^{m-1}e^{ik\pi/m}\right|
 =\left|\frac{1-e^{i\pi}}{1-e^{i\pi/m}}\right|
 =\frac1{\sin(\pi/(2m))}.
\]
Every real projection is therefore bounded by this quantity; continuity
covers the directions for which one of the cosines vanishes. The nodes are symmetric about zero and all
$\cos\theta_{m,j}$ are positive. Their geometric sum is therefore the
positive real number in \eqref{eq:regular-polygon-cosine-sum}, proving
attainment and both identities.
\end{proof}

\subsection{Quadratic Chebyshev lemmas}
\begin{lemma}\label{lem:quadratic-image-support}
For $\kappa>0$ and $\sigma\in\{-1,1\}$, define
$q_{\kappa,\sigma}:\C^2\to\C$ by
\[
 q_{\kappa,\sigma}(x,y):=x^2-y^2-i\sigma\kappa xy.
\]
Its image on $[-1,1]^2$ is the compact convex set
\begin{equation}\label{eq:quadratic-image-set}
 K_\kappa:=\{u+iv\in\C:|u|+v^2/\kappa^2\le1\}.
\end{equation}
Define the support function $h_\kappa:\R\to(0,\infty)$ and its mean $H_\kappa$ by
\begin{equation}\label{eq:quadratic-support-definition}
 h_\kappa(\psi):=\max_{w\in K_\kappa}\operatorname{Re}(e^{-i\psi}w),
 \qquad
 H_\kappa:=\frac1{2\pi}\int_0^{2\pi}h_\kappa(\psi)\,d\psi.
\end{equation}
Then
\begin{equation}\label{eq:quadratic-mean-support}
 H_\kappa=\frac1\pi\left(
 \sqrt{4+\kappa^2}
 +\frac{\kappa^2}{2}\log\frac{2+\sqrt{4+\kappa^2}}{\kappa}
 \right).
\end{equation}
For integers $m\ge1$, set
\begin{equation}\label{eq:full-circle-midpoints}
 \varphi_{m,j}:=\frac{2\pi(j-\tfrac12)}m,\qquad 1\le j\le m.
\end{equation}
Moreover, for every $m\ge1$,
\begin{equation}\label{eq:quadratic-rotated-norm-limit}
 \left|\max_{w_1,\ldots,w_m\in K_\kappa}
 \left|\frac1m\sum_{j=1}^m e^{i\varphi_{m,j}}w_j\right|-H_\kappa\right|
 \le\frac{(2+\kappa)\pi}{m},
 \qquad
 \lim_{m\to\infty}\max_{w_1,\ldots,w_m\in K_\kappa}
 \left|\frac1m\sum_{j=1}^m e^{i\varphi_{m,j}}w_j\right|=H_\kappa.
\end{equation}
\end{lemma}

\begin{proof}
If $u=x^2-y^2$ and $v=-\sigma\kappa xy$, then
\[
 |u|+v^2/\kappa^2=|x^2-y^2|+x^2y^2\le1.
\]
Conversely, for $u+iv\in K_\kappa$, put
\[
 A=\frac{\sqrt{u^2+4v^2/\kappa^2}+u}{2},\qquad
 B=\frac{\sqrt{u^2+4v^2/\kappa^2}-u}{2}.
\]
Both numbers are nonnegative. Since
$v^2/\kappa^2\le1-|u|$, one has
\[
 \sqrt{u^2+4v^2/\kappa^2}\le2-|u|,
 \qquad
 \max\{A,B\}=\frac{\sqrt{u^2+4v^2/\kappa^2}+|u|}{2}\le1.
\]
Also $A-B=u$ and $AB=v^2/\kappa^2$.
Choosing $x^2=A$, $y^2=B$, and the signs so that
$xy=-\sigma v/\kappa$ proves the image formula.  Convexity follows from
convexity of $(u,v)\mapsto |u|+v^2/\kappa^2$.

The set $K_\kappa$ is symmetric about both coordinate axes.  In the first
quadrant, maximizing first over $u$ gives
\[
 h_\kappa(\psi)=\max_{-1\le t\le1}
             \bigl((1-t^2)\cos\psi+\kappa t\sin\psi\bigr).
\]
With $\psi_\kappa:=\arctan(2/\kappa)$, this is
\[
 h_\kappa(\psi)=
 \begin{cases}
 \displaystyle\cos\psi+\frac{\kappa^2\sin^2\psi}{4\cos\psi},
                  &0\le\psi\le\psi_\kappa,\\[5pt]
 \kappa\sin\psi,&\psi_\kappa\le\psi\le\pi/2.
 \end{cases}
\]
Using $\sin^2\psi/\cos\psi=\sec\psi-\cos\psi$,
\begin{align*}
 \int_0^{\pi/2}h_\kappa(\psi)\,d\psi
 &=\int_0^{\psi_\kappa}
 \left(\left(1-\frac{\kappa^2}{4}\right)\cos\psi
       +\frac{\kappa^2}{4}\sec\psi\right)d\psi
   +\kappa\int_{\psi_\kappa}^{\pi/2}\sin\psi\,d\psi\\
 &={\left(1-\frac{\kappa^2}{4}\right)\sin\psi_\kappa
   +\frac{\kappa^2}{4}\log(\sec\psi_\kappa+\tan\psi_\kappa)
   +\kappa\cos\psi_\kappa}.
\end{align*}
{\color{black}Since $\tan\psi_\kappa=2/\kappa$,}
\[
 {\color{black}
 \sin\psi_\kappa=\frac{2}{\sqrt{4+\kappa^2}},\qquad
 \cos\psi_\kappa=\frac{\kappa}{\sqrt{4+\kappa^2}},\qquad
 \sec\psi_\kappa+\tan\psi_\kappa
 =\frac{\sqrt{4+\kappa^2}+2}{\kappa}.}
\]
{\color{black}Substitution gives}
\begin{align*}
 \int_0^{\pi/2}h_\kappa(\psi)\,d\psi
 &=\frac{2-\kappa^2/2+\kappa^2}{\sqrt{4+\kappa^2}}
   +\frac{\kappa^2}{4}\log\frac{2+\sqrt{4+\kappa^2}}{\kappa}\\
 &=\frac{\sqrt{4+\kappa^2}}2
   +\frac{\kappa^2}{4}\log\frac{2+\sqrt{4+\kappa^2}}{\kappa}.
\end{align*}
By the symmetry of $K_\kappa$, $h_\kappa$ has the same integral on each quadrant. Hence
\[
 H_\kappa=\frac2\pi\int_0^{\pi/2}h_\kappa(\psi)\,d\psi,
\]
which proves \eqref{eq:quadratic-mean-support}.

For \eqref{eq:quadratic-rotated-norm-limit}, maximization of a real
projection, followed by maximization in each $w_j$, gives
\[
 \max_{w_j\in K_\kappa}
 \left|\frac1m\sum_{j=1}^m e^{i\varphi_{m,j}}w_j\right|
 =\max_{\psi\in\R}\frac1m\sum_{j=1}^m
                 h_\kappa(\psi-\varphi_{m,j}).
\]
Since $|w|\le2+\kappa$ on $K_\kappa$,
\[
 |h_\kappa(\psi)-h_\kappa(\chi)|
 \le(2+\kappa)|e^{-i\psi}-e^{-i\chi}|
 \le(2+\kappa)|\psi-\chi|.
\]
On each interval of length $2\pi/m$, the distance to its midpoint is at most
$\pi/m$. Hence the midpoint sums differ from $H_\kappa$ by at most
$(2+\kappa)\pi/m$, uniformly in $\psi$. Letting $m\to\infty$
gives the limit in \eqref{eq:quadratic-rotated-norm-limit}.
\end{proof}

\begin{lemma}\label{lem:consecutive-Chebyshev-image}
Let $d\ge2$, $k:=\lfloor d/2\rfloor$, and $\ell:=d-k$.
For $\kappa>0$ and $\sigma\in\{-1,1\}$, define
$Q_{d,\kappa,\sigma}:\C^2\to\C$ by
\begin{equation}\label{eq:consecutive-Chebyshev-polynomial}
 Q_{d,\kappa,\sigma}(x,y)
 :=\frac{T_d(x)-T_d(y)}2
   -\frac{i\sigma\kappa}{2}
       \bigl(T_k(x)T_\ell(y)+T_\ell(x)T_k(y)\bigr).
\end{equation}
Its total degree is at most $d$, and
\begin{equation}\label{eq:consecutive-Chebyshev-image}
 Q_{d,\kappa,\sigma}([-1,1]^2)\subseteq K_\kappa.
\end{equation}
\end{lemma}

\begin{proof}
Write $x=\cos\theta$ and $y=\cos\psi$, where $\theta,\psi\in[0,\pi]$,
and put $A=d\theta/2$, $B=d\psi/2$. Set
\[
 u:=\frac{T_d(x)-T_d(y)}2=\cos^2 A-\cos^2 B,
 \qquad
 v:=\frac{T_k(x)T_\ell(y)+T_\ell(x)T_k(y)}2.
\]
Then $Q_{d,\kappa,\sigma}(x,y)=u-i\sigma\kappa v$.
For $a,b\in[0,1]$, direct calculation in the two cases $a\ge b$ and
$a\le b$ gives
\begin{equation}\label{eq:two-products-control}
 |a-b|+ab\le1,\qquad
 |a-b|+(1-a)(1-b)\le1.
\end{equation}
Taking $a=\cos^2 A$ and $b=\cos^2 B$ gives
\[
 |\cos A\cos B|\le\sqrt{1-|u|},\qquad
 |\sin A\sin B|\le\sqrt{1-|u|}.
\]
If $d=2k$, then $v=\cos A\cos B$, so $|u|+v^2\le1$.

If $d=2k+1$, the addition formulas give
\begin{align}
 v&=\cos A\cos B\cos(\theta/2)\cos(\psi/2)
     -\sin A\sin B\sin(\theta/2)\sin(\psi/2).
 \label{eq:odd-Chebyshev-convex-combination}
\end{align}
The two weights are nonnegative and their sum is
$\cos((\theta-\psi)/2)\le1$. Hence
\[
 |v|\le\sqrt{1-|u|}
 \bigl(\cos(\theta/2)\cos(\psi/2)
       +\sin(\theta/2)\sin(\psi/2)\bigr)
 \le\sqrt{1-|u|}.
\]
Hence $|u|+v^2\le1$, which gives
\eqref{eq:consecutive-Chebyshev-image}; the degree bound follows from
$k+\ell=d$.
\end{proof}

\section{Proof of Theorem~\ref{thm:complexification-identity}}
\label{sec:proof-theorem-A}

\begin{proof}
Fix $s\ge1$, let $0\ne H\in\C[t_1,\ldots,t_s]$ have total degree at most $d$,
put $\N_0:=\{0,1,2,\ldots\}$, and write
\[
 H(t)=\sum_{|\alpha|\le d}c_\alpha t^\alpha,
 \qquad \alpha\in\N_0^s.
\]
For $N\ge d$ and blocks $z^{(j)}\in\C^N$, define the multiaffine polynomial
$H_N:\C^{sN}\to\C$ by
\[
 H_N(z^{(1)},\ldots,z^{(s)})
 :=\sum_{|\alpha|\le d}c_\alpha
       \prod_{j=1}^s U_{N,\alpha_j}(z^{(j)}).
\]
Each $U_{N,k}$ is multiaffine of degree $k$ and satisfies
$U_{N,k}(\zeta,\ldots,\zeta)=\zeta^k$. Thus $H_N$ is multiaffine
of degree at most $d$ in $sN\ge d$ variables and is admissible in
\eqref{eq:uniform-transfer-constant}. For $x^{(j)}\in\{-1,1\}^N$, put
$\tau_j:=m_N(x^{(j)})$. Since every factor
$U_{N,\alpha_j}(x^{(j)})$ and $\tau_j^{\alpha_j}$ has modulus at most one,
the identity
\[
 \prod_{j=1}^s u_j-\prod_{j=1}^s v_j
 =\sum_{\ell=1}^s(u_\ell-v_\ell)
       \prod_{j<\ell}u_j\prod_{j>\ell}v_j
\]
and Lemma~\ref{lem:symmetric-lift-approximation} apply. Set
\[
 E(H):=\sum_{|\alpha|\le d}|c_\alpha|
        \sum_{j=1}^s\alpha_j(\alpha_j-1).
\]
Then
\begin{align}
 |H_N(x)-H(\tau_1,\ldots,\tau_s)|
 &\le \frac1N\sum_{|\alpha|\le d}|c_\alpha|
        \sum_{j=1}^s\alpha_j(\alpha_j-1)\notag\\
 &=\frac{E(H)}N.
 \label{eq:general-block-error}
\end{align}
Hence
\begin{equation}\label{eq:block-lift-cube-bound}
 \|H_N\|_{\{-1,1\}^{sN}}
 \le \|H\|_{[-1,1]^s}+\frac{E(H)}N.
\end{equation}
For every $\zeta=(\zeta_1,\ldots,\zeta_s)\in\T^s$, evaluating the $j$th block
at $(\zeta_j,\ldots,\zeta_j)$ gives
\[
 H_N(\zeta_1,\ldots,\zeta_1;\ldots;\zeta_s,\ldots,\zeta_s)=H(\zeta).
\]
Therefore
\[
 C(d,2)\ge
 \frac{|H(\zeta)|}{\|H\|_{[-1,1]^s}+E(H)/N}.
\]
For fixed $H$ and $s$, $E(H)$ is finite and independent of $N$.
Since $H\ne0$, its norm on $[-1,1]^s$ is positive. Letting $N\to\infty$ yields
\begin{equation}\label{eq:full-cube-lift}
 C(d,2)\ge \frac{|H(\zeta)|}{\|H\|_{[-1,1]^s}}
 \qquad(\zeta\in\T^s).
\end{equation}
Taking the supremum over $\zeta$, then over $H$ and $s$, gives
\[
 C(d,2)\ge
 \sup_{s\ge1}\sup_{\substack{0\ne H\in\C[t_1,\ldots,t_s]\\ \deg H\le d}}
 \frac{\|H\|_{\mathbb T^s}}{\|H\|_{[-1,1]^s}}.
\]

Conversely, every multiaffine polynomial is admissible in
\eqref{eq:Gamma-definition}, and \eqref{eq:multiaffine-real-cube-norm} gives
\[
 \frac{\|P\|_{\mathbb T^N}}{\|P\|_{\{-1,1\}^N}}
 =\frac{\|P\|_{\mathbb T^N}}{\|P\|_{[-1,1]^N}}.
\]
Taking the supremum over $N$ and over nonzero
$P\in\mathcal P_{N,d}^{\mathrm{mult}}$ gives $C(d,2)\le\Gamma_d$.
\end{proof}

\section{Proof of Theorem~\ref{thm:sharp-binary-cost}}
\label{sec:proof-theorem-B}

Set
\[
\rho:=1+\sqrt2.
\]

\par\medskip\noindent\textbf{Step 1. Chebyshev-sum lower bound.}
\par\smallskip\nopagebreak[4]
Let $d\ge2$. On $[-\pi/2,\pi/2]$, define $r_d$ by
\begin{equation}\label{eq:phase-radius}
r_d(\theta):=\cos(\theta/d)+\sqrt{1+\cos^2(\theta/d)}.
\end{equation}
Then
\begin{equation}\label{eq:phase-integral-lower}
C(d,2)\ge\frac14\int_{-\pi/2}^{\pi/2}
\left(r_d(\theta)^d+(-1)^d r_d(\theta)^{-d}\cos(2\theta)\right)\,d\theta.
\end{equation}
Fix $m\ge1$ and use the nodes in \eqref{eq:phase-midpoint-nodes}.
Define $F_{m,d}:\C^m\to\C$ by
\[
 F_{m,d}(t):=\frac1m\sum_{j=1}^m e^{i\theta_{m,j}}T_d(t_j).
\]
Since $T_d([-1,1])=[-1,1]$, its real-cube norm is
\begin{align}
 \|F_{m,d}\|_{[-1,1]^m}
 &=\frac1m\max_{u\in[-1,1]^m}
       \left|\sum_{j=1}^m e^{i\theta_{m,j}}u_j\right|\notag\\
 &=\frac1{m\sin(\pi/(2m))}.
 \label{eq:phase-sum-real-norm}
\end{align}
The last equality is Lemma~\ref{lem:regular-phase-sum}.

Define $w_d:[-\pi/2,\pi/2]\to\C\setminus\{0\}$ and
$z_d:[-\pi/2,\pi/2]\to\T$ by
\begin{equation}\label{eq:phase-Joukowski-points}
 w_d(\theta):=r_d(\theta)e^{i(\pi/2-\theta/d)},
 \qquad
 z_d(\theta):=\frac{w_d(\theta)+w_d(\theta)^{-1}}2.
\end{equation}
Since $r_d(\theta)-r_d(\theta)^{-1}=2\cos(\theta/d)$,
\begin{align}
 4|z_d(\theta)|^2
 &=r_d(\theta)^2+r_d(\theta)^{-2}
       +2\cos(\pi-2\theta/d)\notag\\
 &=2+4\cos^2(\theta/d)-2\cos(2\theta/d)=4.
 \label{eq:phase-points-on-torus}
\end{align}
The Chebyshev--Joukowski identity \cite[Chap.~1]{MasonHandscomb}
\[
 T_d\!\left(\frac{w+w^{-1}}2\right)=\frac{w^d+w^{-d}}2
 \qquad(w\in\C\setminus\{0\})
\]
gives
\begin{equation}\label{eq:phase-aligned-Chebyshev}
 i^{-d}e^{i\theta}T_d(z_d(\theta))
 =\frac12\left(r_d(\theta)^d
        +(-1)^d r_d(\theta)^{-d}e^{2i\theta}\right).
\end{equation}

{\color{black}By \eqref{eq:full-cube-lift}} and
\eqref{eq:phase-sum-real-norm},
\begin{align}
 C(d,2)&\ge m\sin\!\left(\frac\pi{2m}\right)
 \left|F_{m,d}(z_d(\theta_{m,1}),\ldots,z_d(\theta_{m,m}))\right|\notag\\
 &\ge\frac{\sin(\pi/(2m))}{2}\sum_{j=1}^m
 \left(r_d(\theta_{m,j})^d
       +(-1)^d r_d(\theta_{m,j})^{-d}\cos(2\theta_{m,j})\right).
 \label{eq:finite-phase-lower}
\end{align}
The second inequality takes the real part after multiplication by $i^{-d}$
and uses \eqref{eq:phase-aligned-Chebyshev}.
The summand is continuous in $\theta$.  Since
$m\sin(\pi/(2m))\to\pi/2$, the midpoint Riemann sums in
\eqref{eq:finite-phase-lower} converge to the right-hand side of
\eqref{eq:phase-integral-lower} as $m\to\infty$.

\par\medskip\noindent\textbf{Step 2. Explicit estimate from \eqref{eq:phase-integral-lower}.}\par\smallskip\nopagebreak[4]
For each integer $d\ge2$, put
\[
 s_d:=\sqrt{1+\cos^2(\pi/(2d))}.
\]
Then
\begin{equation}\label{eq:exponential-improved-lower}
 C(d,2)\ge\frac\pi4\left(
 \rho^d\exp\!\left(-\frac{\pi^2}{24d s_d}\right)
 -\rho^{-d}\exp\!\left(\frac{\pi^2}{8d s_d}\right)\right).
\end{equation}
For every $d\ge3$,
\begin{equation}\label{eq:simple-improved-lower}
 C(d,2)\ge\frac\pi4\left(1-\frac1{3d}\right)\rho^d.
\end{equation}
For $|\theta|\le\pi/2$, the identity
$\log r_d(\theta)=\operatorname{arsinh}(\cos(\theta/d))$ gives
\begin{align}
 0\le\log\rho-\log r_d(\theta)
 &=\int_{\cos(\theta/d)}^1\frac{du}{\sqrt{1+u^2}}\notag\\
 &\le\frac{1-\cos(\theta/d)}{s_d}
 \le\frac{\theta^2}{2d^2s_d}.
 \label{eq:phase-radius-log-bound}
\end{align}
It follows that
\begin{equation}\label{eq:phase-radius-power-bounds}
 r_d(\theta)^d\ge\rho^d e^{-\theta^2/(2d s_d)},
 \qquad
 r_d(\theta)^{-d}\le\rho^{-d}e^{\pi^2/(8d s_d)}.
\end{equation}
Jensen's inequality \cite[Chap.~III]{HardyLittlewoodPolya} for the normalized measure $d\theta/\pi$ on
$[-\pi/2,\pi/2]$ gives
\begin{align}
 \frac1\pi\int_{-\pi/2}^{\pi/2}e^{-\theta^2/(2d s_d)}\,d\theta
 &\ge\exp\!\left(-\frac1{2d s_d\pi}
                 \int_{-\pi/2}^{\pi/2}\theta^2\,d\theta\right)\notag\\
 &=\exp\!\left(-\frac{\pi^2}{24d s_d}\right).
 \label{eq:phase-radius-integral-bound}
\end{align}
Substitution of \eqref{eq:phase-radius-power-bounds} and
\eqref{eq:phase-radius-integral-bound} into
\eqref{eq:phase-integral-lower}, together with $|\cos(2\theta)|\le1$, gives
\eqref{eq:exponential-improved-lower}.

For $d\ge3$, one has $s_d\ge\sqrt7/2$.  Set
$a:=\pi^2/(12\sqrt7)$.  Since $e^{-a/d}\ge1-a/d$,
\begin{equation}\label{eq:phase-correction-reduction}
 \frac{C(d,2)}{\rho^d}
 \ge\frac\pi4\left(1-\frac ad-\rho^{-2d}e^{3a/d}\right).
\end{equation}
The sequence $d\rho^{-2d}e^{3a/d}$ decreases for $d\ge3$, because the
ratio of consecutive terms is
\[
 \frac{d+1}{d}\rho^{-2}
       \exp\!\left(-\frac{3a}{d(d+1)}\right)
 <\frac4{3\rho^2}<1.
\]
Hence $d\rho^{-2d}e^{3a/d}\le3\rho^{-6}e^a$.  The elementary bounds
\[
 a<\frac{39}{125},\qquad e^a<\frac{11}{8},\qquad
 \rho^6=99+70\sqrt2>197
\]
imply
\begin{equation}\label{eq:phase-correction-rational-check}
 a+3\rho^{-6}e^a
 <\frac{39}{125}+\frac{33}{1576}<\frac13.
\end{equation}
The estimate $a<39/125$ in \eqref{eq:phase-correction-rational-check} follows from $\pi<22/7$ and $\sqrt7>66/25$.
For the second, $a<5/16$ and
\[
 e^x\le1+x+\frac{x^2}{2(1-x/3)}\qquad(0\le x<3)
\]
follow by summing the exponential series and using
$k!\ge2\cdot3^{k-2}$ for $k\ge2$; the displayed upper bound at $x=5/16$
is smaller than $11/8$.
Combining \eqref{eq:phase-correction-reduction} and
\eqref{eq:phase-correction-rational-check} proves
\eqref{eq:simple-improved-lower}.

\par\medskip\noindent\textbf{Step 3. Consecutive Chebyshev lower bound.}\par\smallskip\nopagebreak[4]
For each integer $d\ge3$, put $k:=\lfloor d/2\rfloor$, $\ell:=d-k$, and
\begin{equation}\label{eq:all-degree-radius}
 a_d:=\cos\frac\pi d+\sqrt{1+\cos^2\frac\pi d},\qquad
 b_d:=\rho a_d,\qquad
 \tau_d:=\sqrt{1+\cos^2\frac\pi d}.
\end{equation}
For every $\kappa>0$,
\begin{align}
 \frac{C(d,2)}{\rho^d}
 \ge\frac{2+\kappa}{4H_\kappa}
 \Bigg[&\frac{2+\kappa\cos\bigl((\ell-k)\pi/(2d)\bigr)}{2+\kappa}
       \exp\!\left(-\frac{\pi^2}{6d\tau_d}\right)\notag\\
 &-\frac{\kappa}{2+\kappa}\bigl(b_d^{-k}+b_d^{-\ell}\bigr)
       -b_d^{-d}\Bigg].
 \label{eq:all-degree-finite-bound}
\end{align}
Fix $d\ge3$. {\color{black}Define
\[
\alpha,\gamma:(0,2\pi)\to[-\pi/2,\pi/2],\qquad
\sigma:(0,2\pi)\to\{-1,1\},
\]
by}
\[
 (\alpha(\varphi),\gamma(\varphi),\sigma(\varphi)):=
 \begin{cases}
 (-\varphi/2,\,\pi/2-\varphi/2,\,1),&0<\varphi\le\pi,\\
 (\pi-\varphi/2,\,\pi/2-\varphi/2,\,-1),&\pi<\varphi<2\pi.
 \end{cases}
\]
These functions satisfy
\begin{equation}\label{eq:all-degree-phase-identities}
 |\alpha|,|\gamma|\le\pi/2,\qquad
 \gamma-\alpha=\sigma\pi/2,\qquad
 e^{i\varphi}e^{2i\alpha}=1.
\end{equation}
For $\xi\in[-\pi/2,\pi/2]$, set
\begin{equation}\label{eq:all-degree-evaluation-maps}
 \begin{split}
 \widehat r_d(\xi)&:=\cos(2\xi/d)+\sqrt{1+\cos^2(2\xi/d)},\\
 \widehat w_d(\xi)&:=\widehat r_d(\xi)e^{i(\pi/2+2\xi/d)},\qquad
 \widehat z_d(\xi):=\frac{\widehat w_d(\xi)+\widehat w_d(\xi)^{-1}}2.
 \end{split}
\end{equation}
Thus $\widehat r_d:[-\pi/2,\pi/2]\to[a_d,\rho]$ and
$\widehat w_d:[-\pi/2,\pi/2]\to\C\setminus\{0\}$.
The identity $\widehat r_d-\widehat r_d^{-1}=2\cos(2\xi/d)$ gives
\[
 4|\widehat z_d(\xi)|^2
 =2+4\cos^2(2\xi/d)-2\cos(4\xi/d)=4,
\]
so $\widehat z_d:[-\pi/2,\pi/2]\to\T$.

For the phases $\varphi_{m,j}$ in \eqref{eq:full-circle-midpoints}, write
$\alpha_j=\alpha(\varphi_{m,j})$, $\gamma_j=\gamma(\varphi_{m,j})$, and
$\sigma_j=\sigma(\varphi_{m,j})$. Define $G_{m,d}:\C^{2m}\to\C$ by
\begin{equation}\label{eq:all-degree-rotated-polynomial}
 G_{m,d}(x_1,y_1,\ldots,x_m,y_m)
 :=\frac1m\sum_{j=1}^m e^{i\varphi_{m,j}}
                  Q_{d,\kappa,\sigma_j}(x_j,y_j).
\end{equation}
Its coefficient of $x_1^d$ is $2^{d-2}e^{i\varphi_{m,1}}/m$, so it
has degree $d$. By
Lemma~\ref{lem:consecutive-Chebyshev-image},
\begin{equation}\label{eq:all-degree-real-norm}
 \|G_{m,d}\|_{[-1,1]^{2m}}
 \le\max_{w_j\in K_\kappa}
      \left|\frac1m\sum_{j=1}^m e^{i\varphi_{m,j}}w_j\right|
 \le H_\kappa+\frac{(2+\kappa)\pi}{m}.
\end{equation}
The second inequality is the quantitative midpoint estimate in \eqref{eq:quadratic-rotated-norm-limit}.

Evaluate the $j$th pair of variables at
$(\widehat z_d(\alpha_j),\widehat z_d(\gamma_j))$. To estimate one summand, omit the index
$j$ and put
\[
 R_1:=\widehat r_d(\alpha),\qquad R_2:=\widehat r_d(\gamma),\qquad
 w_1:=\widehat w_d(\alpha),\qquad w_2:=\widehat w_d(\gamma).
\]
For each nonnegative integer $h$, the Chebyshev--Joukowski identity \cite[Chap.~1]{MasonHandscomb} reads
\[
 T_h(\widehat z_d(\alpha))=\frac{w_1^h+w_1^{-h}}2,
 \qquad
 T_h(\widehat z_d(\gamma))=\frac{w_2^h+w_2^{-h}}2.
\]
{\color{black}After multiplication by $i^{-d}e^{i\varphi}$, the leading terms of the
first difference in \eqref{eq:consecutive-Chebyshev-polynomial} become
\[
i^{-d}e^{i\varphi}\frac{w_1^d-w_2^d}{4}
=\frac{R_1^d+R_2^d}{4}.
\]
The definitions of $w_1,w_2$ give}
\[
i^{-d}w_1^d=R_1^d e^{2i\alpha},
\qquad
i^{-d}w_2^d=R_2^d e^{2i\gamma}.
\]
By \eqref{eq:all-degree-phase-identities},
$e^{i\varphi}e^{2i\alpha}=1$ and
$e^{2i\gamma}=e^{2i\alpha}e^{i\sigma\pi}
=-e^{2i\alpha}$, which proves the displayed identity.

For the mixed leading terms, $k+\ell=d$ and
$\gamma-\alpha=\sigma\pi/2$ give
\[
\frac{2(k\alpha+\ell\gamma)}{d}
=\alpha+\gamma
 +\frac{(\ell-k)(\gamma-\alpha)}{d}
=\alpha+\gamma
 +\sigma\frac{(\ell-k)\pi}{2d}.
\]
Also, \eqref{eq:all-degree-phase-identities} implies
\[
e^{i\varphi}e^{i(\alpha+\gamma)}
=e^{i(\gamma-\alpha)}
=\sigma i.
\]
Consequently,
\[
i^{-d}e^{i\varphi}(-i\sigma)w_1^kw_2^\ell
=R_1^kR_2^\ell
 \exp\!\left(i\sigma\frac{(\ell-k)\pi}{2d}\right).
\]
Interchanging $k$ and $\ell$ reverses the remaining phase:
\[
i^{-d}e^{i\varphi}(-i\sigma)w_1^\ell w_2^k
=R_1^\ell R_2^k
 \exp\!\left(-i\sigma\frac{(\ell-k)\pi}{2d}\right).
\]
Taking real parts gives
\begin{align}
&\operatorname{Re}\left(
i^{-d}e^{i\varphi}\frac{-i\sigma\kappa}{8}
       (w_1^kw_2^\ell+w_1^\ell w_2^k)\right)\notag\\
&\qquad=\frac{\kappa}{8}
 \cos\!\left(\frac{(\ell-k)\pi}{2d}\right)
       (R_1^kR_2^\ell+R_1^\ell R_2^k).
\label{eq:all-degree-leading-alignment}
\end{align}

The inverse terms in the first difference have total modulus at most
$a_d^{-d}/2$. In each mixed product, the three terms other than the
leading one have total modulus at most
\[
 \frac14\bigl(\rho^\ell a_d^{-k}+\rho^k a_d^{-\ell}+a_d^{-d}\bigr).
\]
{\color{black}Since the sum of the two mixed products is multiplied by $\kappa/2$,
their discarded terms contribute at most
\[
 \frac{\kappa}{4}
 \bigl(\rho^\ell a_d^{-k}+\rho^k a_d^{-\ell}+a_d^{-d}\bigr).
\]}
Thus the total discarded contribution is bounded by
\begin{equation}\label{eq:all-degree-error}
 E_{d,\kappa}:=
 \frac\kappa4\bigl(\rho^\ell a_d^{-k}+\rho^k a_d^{-\ell}\bigr)
       +\frac{2+\kappa}{4}a_d^{-d}.
\end{equation}
Combining \eqref{eq:all-degree-leading-alignment} and \eqref{eq:all-degree-error} produces a lower bound for
$\operatorname{Re}(i^{-d}G_{m,d})$ at the chosen toral point. Applying
{\color{black}\eqref{eq:full-cube-lift}} and \eqref{eq:all-degree-real-norm}, and then
letting $m\to\infty$ gives
\begin{equation}\label{eq:all-degree-integral-bound}
 C(d,2)\ge\frac1{4H_\kappa}\left[\int_0^{2\pi}\left(R_1^d+R_2^d+\frac\kappa2\cos\!\left(\frac{(\ell-k)\pi}{2d}\right)(R_1^kR_2^\ell+R_1^\ell R_2^k)\right)\frac{d\varphi}{2\pi}-4E_{d,\kappa}\right].
\end{equation}
Here $R_1=\widehat r_d(\alpha(\varphi))$ and
$R_2=\widehat r_d(\gamma(\varphi))$. The integrands are bounded and
piecewise continuous, which justifies the midpoint-sum limit.

To estimate the integral, for $|\xi|\le\pi/2$ one has
\begin{align}
 0\le\log\rho-\log \widehat r_d(\xi)
 &=\int_{\cos(2\xi/d)}^1\frac{du}{\sqrt{1+u^2}}\notag\\
 &\le\frac{1-\cos(2\xi/d)}{\tau_d}
 \le\frac{2\xi^2}{d^2\tau_d}.
 \label{eq:all-degree-log-radius}
\end{align}
Changing variables on the two intervals in the definition of $\alpha$,
and on the full interval in the definition of $\gamma$, gives
\begin{equation}\label{eq:all-degree-phase-second-moments}
 \frac1{2\pi}\int_0^{2\pi}\alpha(\varphi)^2\,d\varphi
 =\frac1{2\pi}\int_0^{2\pi}\gamma(\varphi)^2\,d\varphi
 =\frac{\pi^2}{12}.
\end{equation}
For any nonnegative integers $p,q$ with $p+q=d$,
\eqref{eq:all-degree-log-radius} and Jensen's inequality \cite[Chap.~III]{HardyLittlewoodPolya} therefore imply

\[
 \frac1{2\pi}\int_0^{2\pi}
 \frac{2\bigl(p\alpha(\varphi)^2+q\gamma(\varphi)^2\bigr)}
 {d^2\tau_d}\,d\varphi
 =\frac{\pi^2}{6d\tau_d}.
\]
Hence
\begin{align}
 &\frac1{2\pi}\int_0^{2\pi}
 \widehat r_d(\alpha(\varphi))^p\widehat r_d(\gamma(\varphi))^q\,d\varphi\notag\\*
 &\quad\ge\rho^d\frac1{2\pi}\int_0^{2\pi}
 \exp\!\left(-\frac{2(p\alpha(\varphi)^2+q\gamma(\varphi)^2)}
 {d^2\tau_d}\right)\,d\varphi\notag\\*
 &\quad\ge\rho^d\exp\!\left(-\frac{\pi^2}{6d\tau_d}\right).
 \label{eq:all-degree-Jensen}
\end{align}
Using this estimate for $(p,q)=(d,0),(0,d),(k,\ell),(\ell,k)$ in
\eqref{eq:all-degree-integral-bound}, and using
\[
 \frac{4E_{d,\kappa}}{(2+\kappa)\rho^d}
 =\frac\kappa{2+\kappa}\bigl(b_d^{-k}+b_d^{-\ell}\bigr)+b_d^{-d},
\]
gives \eqref{eq:all-degree-finite-bound}.
\par\medskip\noindent\textbf{Step 4. The choice $\kappa=2/5$.}\par\smallskip\nopagebreak[4]
For $\kappa=2/5$, one has $\kappa/(2+\kappa)=1/6$. It is enough to show
\begin{equation}\label{eq:simple-all-degree-bound}
 C(d,2)\ge\beta\left(1-\frac6{5d}\right)\rho^d
 \qquad(d\ge3).
\end{equation}
In \eqref{eq:quadratic-mean-support}, substitution of $\kappa=2/5$ gives
\[
 \frac{2+\kappa}{4H_\kappa}
 =\frac{15\pi}{2\bigl(5\sqrt{26}+\log(5+\sqrt{26})\bigr)}=\beta,
 \qquad \frac\kappa{2+\kappa}=\frac16.
\]
Set
\[
 x:=\frac{\pi^2}{6d\tau_d},\qquad
 y:=\frac16\left(1-\cos\frac{(\ell-k)\pi}{2d}\right).
\]
Then $x\ge0$, $0\le y\le1$, and $(1-y)e^{-x}\ge1-x-y$. Moreover,
$1-\cos t\le t^2/2$ and $\ell-k\le1$ give
$y\le\pi^2/(48d^2)$. Therefore \eqref{eq:all-degree-finite-bound} yields
\begin{equation}\label{eq:all-degree-elementary-correction}
 \frac{C(d,2)}{\rho^d}
 \ge\beta\left(1-\frac{\pi^2}{6d\tau_d}
                 -\frac{\pi^2}{48d^2}
                 -\frac{b_d^{-k}+b_d^{-\ell}}6-b_d^{-d}\right).
\end{equation}

Suppose first that $d\ge13$. Since $\pi^2<10$,
\[
 \cos(\pi/d)\ge1-\frac{\pi^2}{2d^2}
 >1-\frac5{169}=\frac{164}{169}.
\]
Squaring shows that $\tau_d>39/28$. Since
$\rho>12/5$, $\cos(\pi/d)>19/20$, and $\tau_d>11/8$, it follows that
$b_d>\frac{12}{5}(\frac{19}{20}+\frac{11}{8})>5$.
Here $k\ge6$, $\ell\ge k$, and $d\le2k+1$, so
\begin{align}
 d\left(\frac{b_d^{-k}+b_d^{-\ell}}6+b_d^{-d}\right)
 &\le(2k+1)\left(\frac{5^{-k}}3+5^{-2k}\right)\notag\\
 &\le13\left(\frac1{3\cdot5^6}+\frac1{5^{12}}\right)
 <\frac1{1000}.
 \label{eq:all-degree-tail-check}
\end{align}
The second inequality follows because $(2k+1)5^{-k}$ and
$(2k+1)5^{-2k}$ decrease for $k\ge6$. Therefore
\[
 \frac{\pi^2}{6\tau_d}+\frac{\pi^2}{48d}
 +d\left(\frac{b_d^{-k}+b_d^{-\ell}}6+b_d^{-d}\right)
 <\frac{247}{25}\frac{28}{234}
   +\frac{247}{15600}+\frac1{1000}
 =\frac{21583}{18000}<\frac65.
\]
The first strict inequality uses $\pi^2<(22/7)^2<247/25$.

For $3\le d\le12$, use \eqref{eq:simple-improved-lower}.
The estimates $\sqrt{26}>637/125$ and $\log10>23/10$ give
\[
 5\sqrt{26}+\log(5+\sqrt{26})>\frac{1389}{50}>\frac{250}{9},
 \qquad \beta<\frac{27\pi}{100}.
\]
For the logarithmic bound, the exponential series gives
\[
 e<\sum_{n=0}^{6}\frac1{n!}+\frac{1}{7!}\frac{1}{1-1/8}<\frac{87}{32};
\]
the exponential estimate used above gives
$e^{3/10}\le27/20$. Thus
$e^{23/10}<(87/32)^2(27/20)<10$.
It follows that
\[
 \beta\left(1-\frac6{5d}\right)
 <\frac\pi4\frac{27}{25}\left(1-\frac6{5d}\right)
 \le\frac\pi4\left(1-\frac1{3d}\right).
\]
The final comparison in this display is equivalent to $30d\le361$.
Hence \eqref{eq:simple-all-degree-bound} holds for $3\le d\le12$.
Finally,
\[
\sqrt{26}<\frac{21}{4},
\qquad
5+\sqrt{26}<\frac{41}{4}<\left(\frac83\right)^3<e^3,
\]
Here $e>1+1+1/2+1/6=8/3$ by the exponential series.
Hence
\[
5\sqrt{26}+\log(5+\sqrt{26})
<
\frac{105}{4}+3
=
\frac{117}{4}
<30.
\]
Therefore $\beta>\pi/4$, as asserted in \eqref{eq:asymptotic-beta}.

The upper estimate in \eqref{eq:improved-exponential-sandwich} follows
directly from the classical cube complexification bound
\eqref{eq:Klimek-cube-complexification} and the multiaffine vertex identity
\eqref{eq:multiaffine-real-cube-norm}; equivalently, it is
\eqref{eq:known-binary-remez}. Combining it with
\eqref{eq:simple-all-degree-bound} proves the two-sided estimate of
Theorem~\ref{thm:sharp-binary-cost}.

Dividing by $\rho^d$ and taking lower and upper limits gives
\eqref{eq:improved-normalized-limits}. In particular, the normalized
constants are bounded away from zero by the explicit constant $\beta$.
Taking $d$th roots yields \eqref{eq:sharp-exponential-rate}.

\section*{Notation}
\begingroup
\small
\renewcommand{\arraystretch}{1.0}
\begin{tabular}{@{}p{.22\textwidth}p{.73\textwidth}@{}}
\toprule
Symbol & Definition or reference \\
\midrule
$C(d,2)$ & Optimal binary dimension-free constant in \eqref{eq:uniform-transfer-constant}. \\
$\Gamma_d$ & Unrestricted real-cube complexification constant defined in \eqref{eq:Gamma-definition}. \\
$\mathcal P_{N,d}^{\mathrm{mult}}$ & Multiaffine polynomials on $\C^N$ of total degree at most $d$. \\
$\rho=1+\sqrt2$ & Exponential base used in Theorem~\ref{thm:sharp-binary-cost}. \\
$\beta$ & Constant in \eqref{eq:asymptotic-beta}. \\
$T_d$ & Chebyshev polynomial of the first kind. \\
$U_{n,k}$ & Normalized elementary symmetric polynomial in \eqref{eq:normalized-elementary-symmetric}. \\
$E(H)$ & Coefficient error sum in \eqref{eq:general-block-error}. \\
$m_n$ & Boolean coordinate mean in \eqref{eq:Boolean-empirical-mean}. \\
$\theta_{m,j}$, $\varphi_{m,j}$ & Midpoint phases in \eqref{eq:phase-midpoint-nodes} and \eqref{eq:full-circle-midpoints}. \\
$\alpha$, $\gamma$, $\sigma$ & Phase maps defined immediately before \eqref{eq:all-degree-phase-identities}. \\
$r_d$, $w_d$, $z_d$ & Radius and evaluation maps in \eqref{eq:phase-radius} and \eqref{eq:phase-Joukowski-points}. \\
$K_\kappa$, $h_\kappa$, $H_\kappa$ & Quadratic image, support function, and its mean in \eqref{eq:quadratic-image-set}--\eqref{eq:quadratic-support-definition}. \\
$a_d$, $b_d$, $\tau_d$ & Radius and auxiliary constants in \eqref{eq:all-degree-radius}. \\
$Q_{d,\kappa,\sigma}$ & Chebyshev combination in \eqref{eq:consecutive-Chebyshev-polynomial}. \\
$\widehat r_d$, $\widehat w_d$, $\widehat z_d$ & Evaluation maps in \eqref{eq:all-degree-evaluation-maps}. \\
\bottomrule
\end{tabular}
\endgroup

\section*{Declarations}

\noindent\textbf{Funding.} D. N\'u\~nez-Alarc\'on and J. Santos are supported by Grants No.~406457/2023-9 and No.~403964/2024-5 from the Conselho Nacional de Desenvolvimento Cient\'ifico e Tecnol\'ogico (CNPq, Brazil). These grants correspond, respectively, to Calls CNPq/MCTI No.~10/2023 and MCTI/CNPq No.~16/2024. J. Santos is also supported by CNPq Grant No.~305655/2025-6.

\medskip

\noindent\textbf{Competing interests.}
The authors declare that they have no competing interests.

\medskip
\noindent\textbf{Data availability.}
No data were generated or analyzed in this study.

\medskip
\noindent\textbf{Declaration of generative AI and AI-assisted technologies in the manuscript preparation process.}
The authors conceived and initiated the project and developed the mathematical results. During the preparation of the work, OpenAI's ChatGPT was used to explore alternative approaches, test possible lines of argument, and assist in refining quantitative estimates, exposition, and presentation. All mathematical claims, proofs, and conclusions were reviewed and validated by the authors, who take full responsibility for the content of the manuscript.


\begin{thebibliography}{99}
\small


\bibitem{AronBeauzamyEnflo}
R.~M. Aron, B. Beauzamy, and P. Enflo,
\emph{Polynomials in many variables: real vs. complex norms},
J. Approx. Theory \textbf{74} (1993), no.~2, 181--198,
\href{https://doi.org/10.1006/jath.1993.1060}{doi:10.1006/jath.1993.1060}.


\bibitem{BKSVZ}
L. Becker, O. Klein, J. Slote, A. Volberg, and H. Zhang,
\emph{Dimension-free discretizations of the uniform norm by small product sets},
Invent. Math. \textbf{239} (2025), no.~2, 469--503,
\href{https://doi.org/10.1007/s00222-024-01306-9}{doi:10.1007/s00222-024-01306-9}.

\bibitem{Conway}
J.~B. Conway,
\emph{Functions of One Complex Variable I},
2nd ed., Graduate Texts in Mathematics, vol.~11, Springer-Verlag, New York, 1978,
\href{https://doi.org/10.1007/978-1-4612-6313-5}{doi:10.1007/978-1-4612-6313-5}.

\bibitem{DefantMastyloPerez}
A. Defant, M. Masty\l o, and A. P\'erez,
\emph{On the Fourier spectrum of functions on Boolean cubes},
Math. Ann. \textbf{374} (2019), no.~1--2, 653--680,
\href{https://doi.org/10.1007/s00208-018-1756-y}{doi:10.1007/s00208-018-1756-y}.

\bibitem{DineenBook}
S. Dineen,
\emph{Complex Analysis on Infinite-Dimensional Spaces},
Springer Monographs in Mathematics, Springer-Verlag, London, 1999.

\bibitem{HardyLittlewoodPolya}
G.~H. Hardy, J.~E. Littlewood, and G. P\'olya,
\emph{Inequalities},
2nd ed., Cambridge University Press, Cambridge, 1952.

\bibitem{Klimek}
M. Klimek,
\emph{Metrics associated with extremal plurisubharmonic functions},
Proc. Amer. Math. Soc. \textbf{123} (1995), no.~9, 2763--2770,
{\color{black}\href{https://doi.org/10.1090/S0002-9939-1995-1307539-3}{doi:10.1090/S0002-9939-1995-1307539-3}.}


\bibitem{Lacruz}
M. Lacruz,
\emph{Norms of polynomials and capacities on Banach spaces},
Math. Scand. \textbf{85} (1999), no.~2, 271--277,
\href{https://doi.org/10.7146/math.scand.a-18276}{doi:10.7146/math.scand.a-18276}.

\bibitem{MasonHandscomb}
J.~C. Mason and D.~C. Handscomb,
\emph{Chebyshev Polynomials},
Chapman \& Hall/CRC, Boca Raton, 2003.

\bibitem{MunozSarantopoulosTonge}
{\color{black}
G.~A. Mu\~noz, Y. Sarantopoulos, and A. Tonge,
\emph{Complexifications of real Banach spaces, polynomials and multilinear maps},
Studia Math. \textbf{134} (1999), no.~1, 1--33,
\href{https://doi.org/10.4064/sm-134-1-1-33}{doi:10.4064/sm-134-1-1-33}.
}


\bibitem{NSV2003}
F. Nazarov, M. Sodin, and A. Volberg,
\emph{Local dimension-free estimates for volumes of sublevel sets of analytic functions},
Israel J. Math. \textbf{133} (2003), 269--283,
{\color{black}\href{https://doi.org/10.1007/BF02773070}{doi:10.1007/BF02773070}.}

\bibitem{Ransford}
T.~Ransford,
\emph{Potential Theory in the Complex Plane},
London Mathematical Society Student Texts, vol.~28,
Cambridge University Press, Cambridge, 1995.


\bibitem{Rodriguez2025}
J.~T. Rodr\'iguez,
\emph{On the norm of the complexification of polynomials},
Studia Math. \textbf{282} (2025), 1--23,
\href{https://doi.org/10.4064/sm230615-15-4}{doi:10.4064/sm230615-15-4}.


\bibitem{Siciak1981}
J. Siciak,
\emph{Extremal plurisubharmonic functions in $\mathbb{C}^n$},
Ann. Polon. Math. \textbf{39} (1981), no.~1, 175--211.

\bibitem{SloteRemez}
J. Slote, A. Volberg, and H. Zhang,
\emph{A dimension-free Remez-type inequality on the polytorus},
Discrete Anal. (2025), Paper No.~4, 21~pp.,
\href{https://doi.org/10.19086/da.137968}{doi:10.19086/da.137968}.

\bibitem{Visser}
C. Visser,
\emph{A generalization of Tchebycheff's inequality to polynomials in more than one variable},
Nederl. Akad. Wetensch. Proc. \textbf{49} (1946), 455--456;
Indag. Math. \textbf{8} (1946), 310--311.

\end{thebibliography}
\end{document}